\documentclass{amsart}

\usepackage{amssymb}
\usepackage{amsmath}

\newcommand{\erase}[1]{}

\newtheorem{theorem}{Theorem}[section]
\newtheorem{lemma}[theorem]{Lemma}
\newtheorem{proposition}[theorem]{Proposition}
\newtheorem{corollary}[theorem]{Corollary}

\newtheorem{_algorithm}[theorem]{Algorithm}
\newenvironment{algorithm}{\begin{_algorithm}\rm}{\hfill \rule{3pt}{6pt}\end{_algorithm}}

\newtheorem{_procedure}[theorem]{Procedure}

\newtheorem{_definition}[theorem]{Definition}
\newenvironment{definition}{\begin{_definition}\rm}{\end{_definition}}

\newtheorem{_remark}[theorem]{\it Remark}
\newenvironment{remark}{\begin{_remark}\rm}{\end{_remark}}

\newtheorem{_example}[theorem]{Example}

\numberwithin{equation}{section}
\numberwithin{table}{section}
\numberwithin{figure}{section}

\newcommand{\C}{\mathord{\mathbb C}}
\newcommand{\F}{\mathord{\mathbb F}}
\renewcommand{\P}{\mathord{\mathbb  P}}
\newcommand{\Q}{\mathord{\mathbb  Q}}
\newcommand{\R}{\mathord{\mathbb R}}
\newcommand{\Z}{\mathord{\mathbb Z}}

\newcommand{\FFF}{\mathord{\mathcal F}}
\newcommand{\HHH}{\mathord{\mathcal H}}
\newcommand{\LLL}{\mathord{\mathcal L}}
\newcommand{\NNN}{\mathord{\mathcal N}}
\newcommand{\OOO}{\mathord{\mathcal O}}
\newcommand{\PPP}{\mathord{\mathcal P}}
\newcommand{\XXX}{\mathord{\mathcal X}}

\newcommand{\inj}{\hookrightarrow}
\newcommand{\isom}{\mathbin{\,\raise -.6pt\rlap{$\to$}\raise 3.5pt \hbox{\hskip .3pt$\mathord{\sim}$}\,\;}}

\newcommand{\set}[2]{\{\; {#1} \; \mid \; {#2} \;  \}}
\newcommand{\shortset}[2]{\{ {#1} \,|\, {#2}   \}}

\newcommand{\tensor}{\otimes}

\newcommand{\sprime}{\sp\prime}

\newcommand{\spar}[1]{\sp{(#1)}}
\newcommand{\spprime}{\sp{\prime\prime}}

\newcommand{\sptimes}{\sp{\times}}
\newcommand{\sperp}{\sp{\perp}}

\newcommand{\dual}{\sp{\vee}}
\newcommand{\inv}{\sp{-1}}

\newcommand{\GL}{\mathord{\mathrm {GL}}}

\newcommand{\PGL}{\mathord{\mathrm {PGL}}}
\newcommand{\OG}{\mathord{\mathrm {O}}}

\newcommand{\Image}{\operatorname{\mathrm {Im}}\nolimits}
\newcommand{\Aut}{\operatorname{\mathrm {Aut}}\nolimits}
\newcommand{\Gal}{\operatorname{\mathrm {Gal}}\nolimits}

\newcommand{\pr}{\mathord{\mathrm {pr}}}

\newcommand{\rank}{\operatorname{\mathrm {rank}}\nolimits}
\newcommand{\disc}{\operatorname{\mathrm {disc}}\nolimits}

\newcommand{\mystruth}[1]{\phantom{\hbox{\vrule height #1}}}
\newcommand{\mystruthd}[2]{\phantom{\hbox{\vrule  height #1 depth #2}}}

\newcommand{\Xfs}{X_{56}}
\newcommand{\XFer}{X_{48}}
\newcommand{\hfs}{h_{56}}
\newcommand{\hFer}{h_{48}}

\newcommand{\PP}{\mathord{\bf P}}
\newcommand{\SX}{S_{X}}
\newcommand{\SXp}{S_{X_{56}(P)}}
\newcommand{\TX}{T_{X}}
\newcommand{\tilGFer}{\tilde{G}_{48}}
\newcommand{\tilGfs}{\tilde{G}_{56}}
\newcommand{\intf}[1]{\langle #1\rangle}
\newcommand{\Lines}{\mathord{\FFF}}

\newcommand{\XXXfs}{{\XXX}_{56}}

\newcommand{\XfsP}{{X}_{56}{(P)}}
\newcommand{\LC}{\mathrm{LC}}
\newcommand{\gcds}{\mathrm{gcds}}
\newcommand{\ppar}[1]{(#1)}

\newcommand{\modP}{\hskip -.5pt \bmod{{\hskip -.5pt}P}}
\newcommand{\GramSX}{{\rm G}_{\SX}}
\newcommand{\divbyH}[2]{\overline{#1}{\phantom{\mystruth{8pt}}}^{#2}}

\begin{document}
\title[A smooth quartic surface containing $56$ lines]{On a smooth quartic surface containing $56$ lines
which is  isomorphic as a $K3 $ surface\\ to the Fermat quartic}

\author{Ichiro Shimada}
\address{Department of Mathematics, 
Graduate School of Science, 
Hiroshima University,
1-3-1 Kagamiyama, 
Higashi-Hiroshima, 
739-8526 JAPAN}
\email{shimada@math.sci.hiroshima-u.ac.jp}

\author{Tetsuji Shioda}
\address{Department of Mathematics,  
Rikkyo University,
3-34-1 Nishi-Ikebukuro, Toshima-ku,  
Tokyo, 
171-8501 JAPAN}
\email{shioda@rikkyo.ac.jp}

\thanks{Partially supported by JSPS Grants-in-Aid for Scientific Research (C) No.~25400042 and (S) No.~15H05738}

\begin{abstract}
We give a defining equation of a complex smooth quartic surface containing $56$ lines, 
and investigate its reductions to positive characteristics.
This surface is isomorphic to the complex Fermat quartic surface,
which contains only $48$ lines.
We give the isomorphism explicitly.
\end{abstract}

\maketitle

\section{Introduction}
The complex Fermat quartic surface 
$$
x_1^4+x_2^4+x_3^4+x_4^4=0
$$
is a very interesting surface,
because it lies at the intersection point of the two important classes of algebraic varieties;
$K3$ surfaces and Fermat varieties.
In this paper,
we show that the complex $K3$ surface  underlying the Fermat quartic surface
has  another smooth quartic surface $\Xfs\subset \P^3$
as a projective model.
While the Fermat quartic surface contains only $48$ lines,
our new quartic surface $\Xfs$ contains $56$ lines,
and hence  these two surfaces are  not projectively isomorphic.
We present an explicit  defining equation of  $\Xfs$,
and describe the isomorphism between these two surfaces.
It turns out that the isomorphism $\XFer\isom \Xfs$, the lines on $\XFer$ and $\Xfs$,
and the automorphisms of  $\XFer$ and $\Xfs$ are all defined over the $8$th cyclotomic field $\Q(\zeta)$,
where
$$
\zeta:=\exp(2\pi \sqrt{-1}/8).
$$
We then study the reductions of $\Xfs$ at primes of $\Z[\zeta]$.
\par
In the following,
we denote  the complex Fermat quartic surface by $\XFer$,
and the complex $K3$ surface  underlying $\XFer$ by $X$.
\par 
The existence of  a complex smooth quartic surface containing $56$ lines that is isomorphic to  $\XFer$ 
has been implicitly shown in the paper~\cite{DIS} by Degtyarev, Itenberg and Sert\"oz.
This paper is one of the several works~\cite{DIS, MR3396099, arXiv151201358} on the number of lines 
lying on a smooth quartic surface
that have been done
after the seminal work~\cite{MR3343894}, in which 
Rams and Sch\"utt revised and corrected B.~Segre's classical work~\cite{MR0010431}.
 Degtyarev, Itenberg and Sert\"oz~\cite{DIS} proved the  following theorem.
\begin{theorem}
The number of lines 
lying on a complex  smooth   quartic surface
is either in $\{64, 60, 56, 54\}$  or $\le 52$.
\end{theorem}
The maximum number  $64$ of lines lying on a complex smooth quartic surface is attained by the Schur quartic.
Sch\"utt~\cite{MScomm} has discovered the defining equations of  complex smooth   quartic surfaces
containing $60$ lines.
On the other hand,  in~\cite{DIS},
the transcendental lattices of  complex  smooth  quartic surfaces containing $56$ lines are calculated.
One of these surfaces, which we denote by $\Xfs$,  has the oriented transcendental lattice isomorphic to that of $\XFer$,  
and hence they are isomorphic over the complex number field by the result of~\cite{MR0441982}
 on the classification of complex $K3$ surfaces with Picard number $20$.
 However,
 the defining equation of $\Xfs$ 
and the description of the isomorphism $\XFer\isom \Xfs$ 
have been unknown.
\par
Our main results are as follows.
Let $\PP^3$ and $\P^3$ be the projective spaces with homogeneous coordinates $(x_1: x_2: x_3: x_4)$
and  $(y_1: y_2: y_3: y_4)$, respectively.
Let $\SX$ denote the N\'eron-Severi lattice of $\XFer$,
and let $\hFer\in \SX$ be the class of a hyperplane section of $\XFer\subset \PP^3$.
\begin{theorem}\label{thm:main3}
If $h\in \SX$ is a very ample   class  such that $\intf{h, h}=4$ and $\intf{h, \hFer}=6$,
then the quartic surface  model $X_h$ of $X$ corresponding to $h$ contains exactly $56$ lines.
The number of  very ample   classes $h\in \SX$  satisfying  $\intf{h, h}=4$ and $\intf{h, \hFer}=6$
is $384$.
\end{theorem}
Our quartic surface  $\Xfs\subset \P^3$ corresponds to one of the $384$ very ample classes in Theorem~\ref{thm:main3}.
For a prime $P$ of $\Z[\zeta]$,
let $\kappa_P$ denote the residue field at $P$,  and $\bar{\kappa}_P$ an algebraic closure of $\kappa_P$.
\begin{theorem}\label{thm:main}
{\rm (1)}
We put
$A := -1-2\zeta-2\zeta^3$,
$B:=3+A$,
and
\begin{eqnarray*}
\Psi&:=&y_1^3 y_2+y_1 y_2^3+y_3^3 y_4+y_3 y_4^3\\
&&\phantom{a} +(y_1 y_4+y_2 y_3) (A (y_1 y_3+y_2 y_4)+B (y_1 y_2-y_3 y_4)).
\end{eqnarray*}
Let $\Xfs$ denote the surface in $\P^3$ defined over $\Q(\zeta)$  by the equation $\Psi=0$.
Then the complex  surface  $\Xfs\tensor \C$  
is smooth,  and contains exactly $56$ lines,
each of which is defined over $\Q(\zeta)$.
The projective automorphism group of $\Xfs\tensor \C$ is of order $64$.
Moreover, 
if $P$ is a prime of $\Z[\zeta]$ over a prime integer $p>3$,
then the surface  $\Xfs\tensor \bar{\kappa}_P$   is also smooth, contains exactly $56$ lines, each of which is defined over 
the finite field $\kappa_P$.
\par
{\rm (2)}
There exists  an isomorphism $\XFer\isom \Xfs$ defined over $\Q(\zeta)$
such that the class  $\hfs\in \SX$ of  the pull-back of a hyperplane section of $\Xfs$ 
satisfies  $\intf{\hFer, \hfs}=6$.
This isomorphism induces 
an isomorphism from $\XFer\tensor \bar{\kappa}_P$ to $\Xfs\tensor \bar{\kappa}_P$, 
if $P$ is a prime of $\Z[\zeta]$ over a prime integer $p>3$.
\end{theorem}
\begin{remark}
Since $A=-1-2\sqrt{-2}$, the surface $\Xfs$ is in fact defined over $\Q(\sqrt{-2})$.
\end{remark}
The precise description and 
a geometric characterization  of the class $\hfs\in S_X$ are given in Section~\ref{sec:X56}.
An explicit description of the isomorphism $\XFer\isom \Xfs$   is given in Table~\ref{table:fs},
where the pull-back of the rational functions $y_i/y_j$ on $\Xfs$ are the rational functions $f_i/f_j$ on $\XFer$.
\par
In the study of reductions of $\Xfs$,
we have to 
calculate Gr\"obner bases of ideals in  polynomial rings with coefficients in $\Z[\zeta]$
over  the residue field $\kappa_P$ at infinitely many   primes $P$ of $\Z[\zeta]$.
A simple computational trick  for this task will be given in Section~\ref{sec:reduction}.
For the actual computation, we used GAP~\cite{gap}.
Computational data is available from the author's webpage~\cite{X56compdata}.
\par
The N\'eron-Severi lattices of the Fermat quartic surfaces in characteristic $0$ and in characteristic $3$
were studied by Mizukami and Inose in 1970's.
In particular, they proved that these N\'eron-Severi lattices are generated by the classes of lines.
This fact is crucial for our construction of $\Xfs$.
See~\cite{Mizukami} and Section 6.1 of~\cite{MR2653207}.
\par
Note that we have the following classical theorem due to Matsumura and Monsky~\cite[Theorem 2]{MR0168559}.
\begin{theorem}
If two smooth hypersurfaces of degree $d\ge 3$ in $\P^n$ with $n\ge 3$ are  isomorphic as abstract varieties but not projectively equivalent,
then we have $(d, n)=(4, 3)$. 
\end{theorem}
\noindent
Recently, Oguiso~\cite{Ogui} informed us of 
his method of constructing 
pairs of complex smooth quartic surfaces that are isomorphic as $K3$ surfaces  but are not 
projectively isomorphic.
His result shows in particular  that the graph of the isomorphism $ \XFer\isom\Xfs$
is a complete intersection of $4$ hypersurfaces of bi-degree $(1, 1)$ in $\PP^3\times \P^3$.
\par
%
\par
After the first version of this paper is submitted,
Degtyarev~\cite{newDeg} has proved that,
up to projective equivalence,  the $K3$ surface underlying $\XFer$ has exactly 
three smooth quartic surface models; $\XFer$, $\Xfs$, and its complex conjugate  $\overline{\Xfs}$.
\par
The plan of this paper is as follows.
In~Section~\ref{sec:prelim},
we review the theory of lattices, fix notation, 
and present two algorithms that are used throughout this paper.
In~Section~\ref{sec:Fermat},
we describe the N\'eron-Severi lattice $\SX$ of  $\XFer$ 
by means of the $48$ lines on it.
In~Section~\ref{sec:X56},
we study  very ample line bundles on $\XFer$ that give rise to an isomorphism to $\Xfs$,
and show how to obtain the defining equation of $\Xfs$.
We also compute the projective automorphism group of $\Xfs$.
In~Section~\ref{sec:reduction},
we investigate the reductions of $\Xfs$ at primes of $\Z[\zeta]$.
\par
Thanks are due to Professors Alex Degtyarev, Keiji Oguiso, and Matthias Sch\"utt
for many discussions. 
We also thank the referee for many valuable comments.
\section{Preliminaries on lattices}\label{sec:prelim}
A \emph{lattice} is a free $\Z$-module $L$ of finite rank with a non-degenerate symmetric bilinear form $\intf{\phantom{\cdot}, \phantom{\cdot}}\colon L\times L\to \Z$.
The \emph{orthogonal group} $\OG(L)$ of $L$  acts on $L$ from the \emph{right}.
The \emph{dual lattice} $L\dual$ of $L$ is a submodule of $L\tensor \Q$ consisting of vectors $v\in L\tensor \Q$ such that
$\intf{v, x}\in \Z$ holds for any $x\in L$.
The \emph{discriminant group}  $\disc(L)$ of $L$ is defined to be $L\dual/L$.
A lattice $L$  is \emph{unimodular} if  $\disc(L)$ is trivial.
A lattice $L$ is \emph{even} if $\intf{v, v}\in 2\Z$ holds for any $v\in L$.
Let $L$ be an even lattice.
The  $\Q$-valued  symmetric bilinear form on $L\dual$ that extends $\intf{\phantom{\cdot}, \phantom{\cdot}}\colon L\times L\to \Z$
defines  a finite quadratic form
$q_L\colon \disc (L) \to \Q/2\Z$,
which is called the \emph{discriminant form} of $L$.
Let $\OG(q_L)$ denote the automorphism group of the finite quadratic form $q_L$,
which acts on $\disc(L)$ from the right.
Then we have a natural homomorphism 
$$
\eta_L\colon \OG(L)\to \OG(q_L).
$$
See~\cite{MR525944} for applications of  the theory of discriminant forms.
\par
A lattice $L$ of rank $n>1$ is \emph{hyperbolic} if the signature of the real quadratic space $L\tensor \R$ is $(1, n-1)$.
We have the following algorithms. See~\cite{MR3166075} for details. 
\begin{algorithm}\label{algo:AffES}
Let $M$ be a free $\Z$-module  of finite rank $n>1$ with a $\Q$-valued non-degenerate symmetric bilinear form
$\intf{\phantom{a}, \phantom{a}}\colon M\times M\to \Q$
such that $M\tensor \R$ is of signature $(1,n-1)$.
(For example, $M$ is a hyperbolic lattice or the  dual lattice of a hyperbolic lattice.)
Let $h\in M$ be a vector such that  $\intf{h, h}>0$.
Then, for given rational numbers $a$ and $b$, we can make the list of all vectors $x$ of $M$ 
that satisfy
$\intf{h, x}=a$ and 
$\intf{x, x}=b$.
\end{algorithm}
\begin{algorithm}\label{algo:Sep}
Let $L$ be a hyperbolic lattice,
and let $h, h\sprime$ be vectors of $L$  that satisfy $\intf{h, h}>0$, $\intf{h\sprime, h\sprime}>0$ and $\intf{h, h\sprime}>0$.
Then, for a negative integer $d$, we can make 
the list of all vectors $x$ of $L$ 
that satisfy
$\intf{h, x}>0$, 
$\intf{h\sprime,  x}<0$ and $\intf{x,  x}=d$.
\end{algorithm}
\begin{remark}
These algorithms are based on an algorithm of positive-definite quadratic forms 
described in Section 3.1 of~\cite{MR3166075}.
This algorithm can be made much faster by means of the lattice reduction basis~\cite{MR682664}.
See  Section 2.7 of~\cite{MR1228206}.
\end{remark}
Let $L$ be an even hyperbolic lattice,
and let $\PPP(L)$ be one of the two connected components of $\shortset{x\in L\tensor\R}{\intf{x,x}>0}$,
which we call a \emph{positive cone} of $L$.
A vector $r\in L$ is called a \emph{$(-2)$-vector} if $\intf{r, r}=-2$ holds.
For a $(-2)$-vector $r$,
we have a \emph{reflection} $s_r\colon x\mapsto x+\intf{x, r}\, r$
in the hyperplane 
$$
(r)\sperp :=\set{x\in L\tensor \R}{\intf{r, x}=0}.
$$
Each $s_r\in \OG(L)$ acts on $\PPP(L)$.
The reflections $s_r$ with respect to all $(-2)$-vectors $r$ generate a subgroup $W(L)$ of $\OG(L)$.
The closure in $\PPP(L)$ of a connected component of
$$
\PPP(L)\;\setminus\; \bigcup\; (r)\sperp
$$
is called a \emph{standard fundamental domain} of the action of $W(L)$ on $\PPP(L)$.
Let $\NNN$ be a standard fundamental domain,
and let $h$ be an element of $\NNN\cap L$ such that $h\notin (r)\sperp$ for any $(-2)$-vectors $r$.
Then Algorithm~\ref{algo:Sep} applied to $d=-2$
provides us with a method to determine whether a given vector $h\sprime \in \PPP(L)\cap L$
is contained in $\NNN$ or not.
\par
The \emph{N\'eron-Severi lattice} $S_Y$ of an algebraic  $K3$ surface $Y$ is the $\Z$-module of numerical equivalence classes of divisors on $Y$
with the intersection pairing.
The lattice $S_Y$  is even,
and if its rank is $>1$, it is hyperbolic.
The class of a curve $C$ on $Y$ is denoted by $[C]\in S_Y$.
Suppose that $\rank \SX>1$.
It is well-known that 
the \emph{nef cone} 
$$
\set{x\in S_Y\tensor\R}{\textrm{$\intf{x, [C]}\ge 0$ for any curve $C$ on $Y$}}
$$
is   a  standard fundamental domain  of the action of $W(\SX)$.
When $Y$ is defined over $\C$, 
the second cohomology group $H^2(Y, \Z)$ of a complex $K3$ surface $Y$ with the cup product 
is an even, unimodular lattice of signature $(3, 19)$
containing  $S_Y$ as a primitive sublattice.
\section{The Fermat quartic surface $\XFer$}\label{sec:Fermat}
The complex Fermat quartic surface is denoted by $\XFer\subset \PP^3$.
The complex surface underlying $\XFer$ is 
simply denoted by $X$.
We describe the N\'eron-Severi lattice $\SX$ of $X$
in terms of the lines on $\XFer$.
In particular,
we study  the condition for an isometry $g\in \OG(\SX)$ of $\SX$ to extend to a Hodge isometry of 
 $H^2(X, \Z)$.
\par
The \emph{transcendental lattice} $\TX$ of $X$ is defined to be the orthogonal complement of $\SX$ in $H^2(X, \Z)$.
It is known~\cite{Mizukami},~\cite{MR0441982} that 
 $\SX$ 
is of rank  $20$, and $\TX$ is isomorphic to the positive-definite lattice 
\begin{equation}
T:=\left[
\begin{array}{cc} 8 & 0 \\ 0 & 8
\end{array}
\right]
\label{eq:GramT}
\end{equation}
of rank $2$.
By~\cite{MR525944}, 
we have an anti-isometry $q_{S_X}\isom -q_{T_X}$ of discriminant forms,  and hence $|\disc (S_X)|=64$. 
\par
The surface $\XFer$   contains exactly $48$ lines.
These lines  are labelled by the tags 
$[i, [\mu,  \nu]]$ in the following way,
where $i$ is an integer satisfying $1< i\le 4$,
and $\mu$ and $\nu$ are positive odd integers $\le 7$.
Let $j$ and $k$ be integers such that $j<k$ and $\{1, i, j, k\}=\{1,2,3,4\}$.
Then the line  on $\XFer$ labelled by  ${[i, [\mu, \nu]]}$ is defined by
$$
x_1+\zeta^{\mu} x_{i}=0, \quad x_{j}+\zeta^{\nu} x_{k}=0.
$$
All lines on $\XFer$ are obtained  in this way.
Following~\cite{MR0011003}, 
we call a  point $Q\in \XFer$  a \emph{$\tau$-point} if 
the intersection   $\XFer\cap T_Q(\XFer)$  of $\XFer$ and   the tangent plane  $T_Q(\XFer)\subset \PP^3$  to $\XFer$ at $Q$
consists of four lines passing though $Q$.
There exist exactly $24$  $\tau$-points,
and each line on $\XFer$ contains exactly two $\tau$-points.
If three distinct lines on $\XFer$ have a common point $Q$,
then $Q$ is a $\tau$-point,
and hence these three lines  
are coplanar.
The converse is also true; 
if three distinct lines on $\XFer$ are coplanar,
then they have a common point, which is a $\tau$-point.
\par
We choose the  $20$ lines $l_1, \dots, l_{20}$ labelled by the following tags respectively:
\begin{eqnarray*}
&&[ 2, [ 1, 1 ] ], \;[ 2, [ 1, 3 ] ], \;[ 2, [ 1, 5 ] ], \;[ 2, [ 1, 7 ] ], \\
&& [ 2, [ 3, 1 ] ], \;[ 2, [ 3, 3 ] ], \;[ 2, [ 3, 5 ] ], \;[ 2, [ 5, 1 ] ], \;[ 2, [ 5, 3 ] ], 
  [ 2, [ 5, 5 ] ], \;[ 3, [ 1, 1 ] ], \;[ 3, [ 1, 3 ] ], \\
  &&[ 3, [ 1, 5 ] ], \;[ 3, [ 3, 1 ] ], \;[ 3, [ 3, 3 ] ], \;[ 3, [ 3, 5 ] ], \;[ 4, [ 1, 1 ] ], \;[ 4, [ 1, 3 ] ], 
  [ 4, [ 1, 5 ] ], \;[ 4, [ 3, 1 ] ].
\end{eqnarray*}
\begin{table}%
{\small%
$$
\renewcommand{\arraycolsep}{.7mm}
\left[ \begin {array}{cccccccccccccccccccc} 
-2&1&1&1&1&0&0&1&0&0&1&0&0&0&1&0&0&0&0&1
\\ \noalign{\medskip}1&-2&1&1&0&1&0&0&1&0&0&1&0&0&0&1&0&0&1&0
\\ \noalign{\medskip}1&1&-2&1&0&0&1&0&0&1&0&0&1&0&0&0&0&1&0&0
\\ \noalign{\medskip}1&1&1&-2&0&0&0&0&0&0&0&0&0&1&0&0&1&0&0&0
\\ \noalign{\medskip}1&0&0&0&-2&1&1&1&0&0&0&0&0&1&0&0&0&0&1&0
\\ \noalign{\medskip}0&1&0&0&1&-2&1&0&1&0&1&0&0&0&1&0&0&1&0&0
\\ \noalign{\medskip}0&0&1&0&1&1&-2&0&0&1&0&1&0&0&0&1&1&0&0&0
\\ \noalign{\medskip}1&0&0&0&1&0&0&-2&1&1&0&0&1&0&0&0&0&1&0&0
\\ \noalign{\medskip}0&1&0&0&0&1&0&1&-2&1&0&0&0&1&0&0&1&0&0&0
\\ \noalign{\medskip}0&0&1&0&0&0&1&1&1&-2&1&0&0&0&1&0&0&0&0&1
\\ \noalign{\medskip}1&0&0&0&0&1&0&0&0&1&-2&1&1&1&0&0&1&0&0&0
\\ \noalign{\medskip}0&1&0&0&0&0&1&0&0&0&1&-2&1&0&1&0&0&1&0&1
\\ \noalign{\medskip}0&0&1&0&0&0&0&1&0&0&1&1&-2&0&0&1&0&0&1&0
\\ \noalign{\medskip}0&0&0&1&1&0&0&0&1&0&1&0&0&-2&1&1&0&1&0&1
\\ \noalign{\medskip}1&0&0&0&0&1&0&0&0&1&0&1&0&1&-2&1&0&0&1&0
\\ \noalign{\medskip}0&1&0&0&0&0&1&0&0&0&0&0&1&1&1&-2&0&0&0&0
\\ \noalign{\medskip}0&0&0&1&0&0&1&0&1&0&1&0&0&0&0&0&-2&1&1&1
\\ \noalign{\medskip}0&0&1&0&0&1&0&1&0&0&0&1&0&1&0&0&1&-2&1&0
\\ \noalign{\medskip}0&1&0&0&1&0&0&0&0&0&0&0&1&0&1&0&1&1&-2&0
\\ \noalign{\medskip}1&0&0&0&0&0&0&0&0&1&0&1&0&1&0&0&1&0&0&-2
\end {array} \right] 
$$
}%
\caption{Gram matrix $\GramSX$ of $\SX$}\label{table:GramSX}%
\end{table}%
Their intersection matrix  $\GramSX$ is given in Table~\ref{table:GramSX}.
Since the determinant of this  matrix is $-64=-|\disc (S_X)|$,
we see that the classes of these $20$ lines form a basis of  $\SX$,
and the matrix $\GramSX$  is the Gram matrix of $\SX$ with respect to this basis.
From now on,
every vector of $\SX$ is written  as a \emph{row} vector with respect to this basis.
Since the four lines $l_1, \dots, l_{4}$  are on the  plane $x_1+\zeta x_2=0$, 
the class $\hFer$  of the  hyperplane section of $\XFer\inj \PP^3$  is given by 
$$
\hFer=[ 1, 1, 1, 1, 0, 0, 0, 0, 0, 0, 0, 0, 0, 0, 0, 0, 0, 0, 0, 0 ].
$$
By Riemann-Roch theorem,
the set of classes $[\ell]$ of lines $\ell$ on $\XFer$ is equal to 
$$
\Lines_{48}:=\set{r\in \SX}{\intf{r, r}=-2, \;\; \intf{r, \hFer}=1}.
$$
This set can be calculated by Algorithm~\ref{algo:AffES}.
The class of each line 
 is also computed  from the intersection numbers with $l_1, \dots, l_{20}$.
We have a basis of  $\SX\dual$  dual to the fixed basis $[l_1], \dots, [l_{20}]$ of $\SX$.
To distinguish the vector representation with respect to the non-dual basis $[l_1], \dots, [l_{20}]$ of $\SX$
and that with respect to the dual basis of $\SX\dual$,
we put  a superscript ${}\dual$  on the dual representation. 
Thus we have a relation 
$$
x\dual=x \, \GramSX
$$
between the non-dual vector representation $x \in \Q^{20}$ of an element  $v\in \SX\tensor\Q$ 
and the dual representation $x\dual$ of $v$.
\par
Consider the  following vectors of $\SX\dual$:
\begin{eqnarray*}
 s_1 &:=&  [ 3, 1, 2, 2, 1, 3, 2, 2, 2, 2, 2, 3, 1, 2, 1, 2, 2, 1, 3, 1 ]\dual, \\
 s_2 &:=&   [ 1, 3, 1, 1, 1, 1, 3, 2, 1, 0, 1, 1, 2, 2, 3, -1, 1, 2, 0, 2 ] \dual.
\end{eqnarray*}
Then the elements 
$$
\sigma_1:=s_1 \bmod \SX, \;\; 
\sigma_2:=s_2 \bmod \SX
$$
of  $\disc (\SX)$
form a basis of  $\disc (\SX)\cong (\Z/8\Z)^2$,
under which the discriminant form $q_{\SX}$ is given by the matrix
$$
\frac{1}{8}\left[\begin{array}{cc} 11 & 5\\ 5 & 14\end{array}\right],
$$
where the diagonal components are in $\Q/2\Z$ and the off-diagonal components are in $\Q/\Z$.
Let $P$ be the $20\times 2$ matrix
$$
\left[ \begin {array}{cccccccccccccccccccc} 7&2&5&6&0&6&6&7&2&7&6&4&6
&2&4&2&4&0&4&0\\ \noalign{\medskip}0&5&3&2&7&6&3&1&7&6&0&6&2&0&2&6&4&4
&4&4\end {array} \right]\hskip -7pt {\phantom{\Bigl(}}^T 
$$
with components in $\Z/8\Z$.
Then the  quotient homomorphism  
$\SX\dual \to \disc(\SX)$ is given by $x\dual \mapsto x\dual P$
with respect to the basis $\sigma_1, \sigma_2$ of $\disc(\SX)$.
When we are given an element $g$ of $\OG(\SX)$ as a $20\times 20$ matrix $R_g$
with respect to the  basis $[l_1], \dots, [l_{20}]$ of $\SX$,
the automorphism  $\eta_{\SX}(g)$  of $q_{\SX}$ induced by $g$ is represented  with respect to the basis $\sigma_1, \sigma_2$  by the $2\times 2$ matrix 
\begin{equation}\label{eq:GRGP}
\left[\begin{array}{c} s_1 \\ s_2 \end{array}\right] \GramSX\inv R_g \GramSX P
\end{equation}
with components in $\Z/8\Z$.
\par
Let $t_1$ and $t_2$ be a basis of the lattice
$T$ under which the Gram matrix is given in~\eqref{eq:GramT},
and let $\C\omega$ be a totally isotropic subspace of $T\tensor \C$.
We have
$$
T\tensor \C=\C\omega\oplus \C\bar{\omega}.
$$
Let $t_1\dual, t_2\dual$ be the basis of $T\dual$  dual to the basis  $t_1, t_2$ of $T$.
Then the elements 
$\tau_1:=t_1\dual \bmod T$ and  $\tau_2:=t_2\dual \bmod T$
form a basis of  $\disc (T)\cong (\Z/8\Z)^2$,
under which  $q_{T}$ is given by the matrix
$$
\frac{1}{8}\left[\begin{array}{cc} 1& 0\\ 0& 1\end{array}\right].
$$
As above, we can calculate the natural homomorphism $\eta_T\colon \OG(T)\to \OG(q_T)$ explicitly.
It is easy to see that $\OG(T)$ is of order $8$,
$\OG(q_T)$ is of order $16$,
and $\eta_T$ is injective.
Moreover,
we see that the group 
$$
\tilde{\Gamma}_T:=\set{g\in \OG(T)}{ \C\omega^g=\C\omega}
$$
is of order $4$.
Since $T$ is isomorphic to the transcendental lattice $T_X$,
there exists an isomorphism $q_{S_X}\cong -q_T$ by~\cite{MR525944}.
In fact, since $\OG(q_T)$ is of order $16$, 
there exist exactly $16$ isomorphisms  from $q_{S_X}$ to $-q_T$.
For an isomorphism $\varphi\colon q_{S_X}\isom -q_T$,
let $\varphi_*\colon \OG(q_{S_X})\isom \OG(q_T)$ be the induced isomorphism.
It turns out that 
the subgroup 
$$
\Gamma_{\SX}:=\set{\gamma \in \OG(q_{\SX})}{\varphi_*(\gamma) \in \eta_T(\tilde\Gamma_T)}
$$
of $\OG(q_{\SX})$ does \emph{not} depend on the choice of $\varphi$;
we have 
\begin{equation}\label{eq:GammaSX}
\Gamma_{\SX}=\left\{\;\;
\left[ \begin {array}{cc} 1&0\\ \noalign{\medskip}0&1\end {array}
 \right], 
 \left[ \begin {array}{cc} 3&3\\ \noalign{\medskip}2&5\end {array}
 \right], 
 \left[ \begin {array}{cc} 5&5\\ \noalign{\medskip}6&3\end {array}
 \right], 
 \left[ \begin {array}{cc} 7&0\\ \noalign{\medskip}0&7\end {array}
 \right] \;\;\right\} \;\; \subset\;\; \GL_2(\Z/8\Z).
 \end{equation}
Note that an isometry $\tilde{g}$ of the lattice $H^2(X, \Z)$ is a Hodge isometry
if and only if $\tilde{g}$ preserves $T_X$ and its orientation.
If $\tilde{g}$ preserves $T_X$,
then the  orientation of $T_X$ is preserved if and only if $\tilde{g}|_{T_X}$ belongs to $\tilde{\Gamma}_T$
under an (and hence any)  isometry $T\cong T_X$.
Hence, 
by~\cite{MR525944}, 
we see that an isometry  $g\in \OG(\SX)$ extends to a Hodge  isometry  of $H^2(X, \Z)$ if and only if 
\begin{equation}\label{eq:preserveomegacond}
\eta_{\SX}(g)\in \Gamma_{\SX}.
\end{equation}
This condition can be checked computationally using~\eqref{eq:GRGP}~and~\eqref{eq:GammaSX}.
\par
We let the automorphism group $\Aut(X)$  act on  $X$ from the left, and act on $\SX$ from the right by the pull-back.
The following facts can be checked by direct computation by means of  the data we have prepared so far.
We consider the subgroup 
$$
\Aut(\XFer):=\set{\gamma\in \PGL_4(\C)}{ \gamma(\XFer)=\XFer}
$$
on $\Aut(X)$, which is known to be of order $1536$
and generated by the permutations of coordinates $x_1, \dots, x_4$ and the scalar-multiplications by $\zeta^2$ of coordinates.
We denote by 
$$
G_{48}:=\Image (\Aut(\XFer)\to \OG(\SX))
$$
the image of $\Aut(\XFer)$ in $\OG(\SX)$ by the natural representation, which is injective.
Since the set $\FFF_{48}$ of classes of lines on $\XFer$ spans $\SX$, 
the stabilizer subgroup 
$$
\tilGFer:=\set{g\in \OG(\SX)}{\hFer^g=\hFer}
$$
of $\hFer$  
is isomorphic to the group of permutations of $\FFF_{48}$ that preserve the intersection numbers.
The mapping 
$$
g\mapsto ([l_1]^g, \dots, [l_{20}]^g)
$$
gives a bijection from $\tilGFer$ to the set of  ordered lists $([l_1\sprime], \dots, [l_{20}\sprime])$ 
of elements of $\FFF_{48}$ that satisfy
$$
\intf{[l_i\sprime], [l_j\sprime]}=\intf{[l_i], [l_j]}\;\; \textrm{for all $i, j=1,\dots, 20$}.
$$
We calculate the set of all these  $([l_1\sprime], \dots, [l_{20}\sprime])$  
by the  standard backtrack program~(see~\cite{MR0373371} for the meaning of the backtrack program), 
and calculate $\tilGFer$ as a list of elements of $\OG(\SX)$.
Since each line on $\XFer$ is defined over $\Q(\zeta)$, 
the Galois group
$\Gal(\Q(\zeta)/\Q)\cong (\Z/2\Z)^2$ also acts on $\FFF_{48}$ preserving the intersection numbers,
and hence acts on $\SX$.
It turns out that  $\tilGFer$ is of order $6144$ and is generated by $G_{48}$ and $\Gal(\Q(\zeta)/\Q)$.
By Torelli theorem~\cite{MR0284440}, the subgroup   $G_{48}$ of $\tilGFer$
 consists of elements $g\in \tilGFer$ that satisfy the period-preserving condition~\eqref{eq:preserveomegacond}.
\par
We review the result of B.~Segre~\cite{MR0011003}
on the set of pairs of lines on $\XFer$.
The group $\Aut(\XFer)$ acts on the $48$ lines transitively.
Let $\PPP_{i}$ be the set of  pairs of intersecting lines on $\XFer$,
and let $\PPP_{d}$ be the set of  pairs of disjoint  lines on $\XFer$.
The orbit decompositions $\PPP_{i}=o_1\sqcup o_2 \sqcup o_3$ 
and $\PPP_{d}=o_4\sqcup \dots \sqcup o_8$  of these sets by the action of   $\Aut(\XFer)$ are as follows:
$$
\begin{array}{lllllll}
\{\;\; [ 2, [ 1, 1 ] ],\; [ 2, [ 1, 5 ] ]  \;\;\} &\in& o_{ 1 }&\subset &\PPP_{i}, &|o_{ 1 }|= 48 ,\\ 
\{\;\; [ 2, [ 1, 1 ] ],\; [ 2, [ 1, 3 ] ]  \;\;\} &\in& o_{ 2 }&\subset &\PPP_{i}, &|o_{ 2 }|= 96 ,\\ 
\{\;\; [ 2, [ 1, 1 ] ],\; [ 3, [ 1, 1 ] ]  \;\;\} &\in& o_{ 3 }&\subset &\PPP_{i}, &|o_{ 3 }|= 192 , \mystruthd{0pt}{6pt}\\ 
\hline 
\{\;\; [ 2, [ 1, 1 ] ],\; [ 2, [ 5, 5 ] ]  \;\;\} &\in& o_{ 4 }&\subset &\PPP_{d}, &|o_{ 4 }|= 24 , \mystruthd{12pt}{0pt}\\ 
\{\;\; [ 2, [ 1, 1 ] ], \;[ 2, [ 3, 3 ] ]  \;\;\} &\in& o_{ 5 }&\subset &\PPP_{d}, &|o_{ 5 }|= 96 ,\\ 
\{\;\; [ 2, [ 1, 1 ] ],\; [ 2, [ 3, 5 ] ]  \;\;\} &\in& o_{ 6 }&\subset &\PPP_{d}, &|o_{ 6 }|= 96 ,\\ 
\{\;\; [ 2, [ 1, 1 ] ],\; [ 3, [ 1, 5 ] ]  \;\;\} &\in& o_{ 7 }&\subset &\PPP_{d}, &|o_{ 7 }|= 192 ,\\ 
\{\;\; [ 2, [ 1, 1 ] ], \;[ 3, [ 1, 3 ] ]  \;\;\} &\in& o_{ 8 }&\subset &\PPP_{d}, &|o_{ 8 }|= 384.
\end{array}
$$
For each orbit $o_i$,
we define an $8\times 8$ matrix $A(o_i)=(a_{jk})$ as follows.
Let $\{\ell, \ell\sprime\}$ be a pair in  $o_i$.
We put 
$$
a_{jk}:=\textrm{the number of lines $\ell\spprime$ such that $\{\ell, \ell\spprime \}\in o_j$ and $\{\ell\sprime, \ell\spprime \}\in o_k$.}
$$
The $3\times 3$ upper-left part $A\sprime(o_i)=(a_{jk})_{1\le j, k\le 3}$
of each of these matrices are given in Table~\ref{table:Ai}.
%
\begin{remark}
Let $\{\ell, \ell\sprime\}$  be a pair of intersecting lines.
Then $\ell$ and $\ell\sprime$ intersect at a $\tau$-point
if and only if  $\{\ell, \ell\sprime\}$  belongs to  $o_1$ or to $o_2$.
\end{remark}
\begin{table}{\footnotesize%
$$ 
A\sprime(o_1)=\left[ \begin {array}{ccc} 
0&0&0\\ 
  0&2&0\\ 
  0&0&0
\end {array} \right],
\qquad
A\sprime(o_2)= \left[ \begin {array}{ccc} 
0&1&0\\ 
  1&0&0\\ 
  0&0&0
  \end {array} \right],
  \quad
  A\sprime(o_3)= \left[ \begin {array}{ccc} 
0&0&0\\ 
  0&0&0\\ 
  0&0&0
 \end {array} \right],
 $$
 \medskip
$$
 A\sprime(o_4)=\left[ \begin {array}{ccc} 
 2&0&0\\
  0&0&0\\
  0&0&8\
  \end {array} \right],
\quad
A\sprime(o_5)=\left[ \begin {array}{ccc} 
 0&0&0\\
  0&2&0\\
  0&0&4
\end {array} \right],
\quad
A\sprime(o_6)=\left[ \begin {array}{ccc} 
 0&1&0\\ 
  1&0&0\\ 
  0&0&0
\end {array} \right],
$$
\par
\smallskip
$$
A(o_7)= \left[ \begin {array}{ccc} 
0&0&2\\ 
  0&0&0\\ 
  2&0&0
\end {array} \right],
\quad
A\sprime(o_8)=\left[ \begin {array}{ccc} 
 0&0&0\\ 
  0&0&2\\ 
  0&2&2
\end {array} \right].
\quad
\phantom{
A\sprime(o_8)=\left[ \begin {array}{ccc} 
 0&0&0\\ 
  0&0&2\\ 
  0&2&2
\end {array} \right] 
}
$$ 
}
\caption{Matrices $A\sprime(o_i)$}\label{table:Ai}\end{table}
\erase{
\par
The orbit $o_4$ of size $24$ plays an important role in the next section.
Let $\{\ell_1, \ell_2\}$ be a member of  $o_4$.
Then there exist exactly $10$ lines $n_1, n_2$, $m_1, \dots, m_8$  that intersect both of $\ell_1$ and $\ell_2$.
These $10$ lines are mutually disjoint.
The intersection point of $\ell_i$  and $n_j$  is a $\tau$-point for $j=1, 2$,
while the intersection point of $\ell_i$  and $m_j$  is not a $\tau$-point for $j=1, \dots, 8$.
}
\section{The quartic surface $\Xfs$}\label{sec:X56}
For $v\in S_X$, let $\LLL_v\to X$ be a line bundle whose class is $v$.
We say that $h\in\SX$ is a \emph{polarization of degree $4$}
if $\intf{h,h}=4$ and the complete linear system $|\LLL_h|$ is fixed-component free.
By~\cite{MR0364263}, if $h$ is a polarization of degree $4$, then $|\LLL_h|$ is base-point free
and defines a morphism  $\Phi_{h} \colon X\to \P^3$.
We say that a  polarization $h$ of degree $4$ is \emph{very ample}
if $\Phi_{h}$ is an embedding.
\begin{theorem}\label{thm:pol}
A class $h\in \SX$ with $\intf{h, h}=4$ is a very ample polarization of degree $4$ if and only if the following hold:
\begin{itemize}
\item[(a)] $\intf{h, \hFer}>0$, 
\item[(b)] $\set{r\in \SX}{\intf{r,r}=-2, \;\intf{r, \hFer}>0, \; \intf{r, h}<0}$ is empty, 
\item[(c)] $\set{e\in \SX}{\intf{e,e}=0, \;\intf{e, h}=1}$ is empty, 
\item[(d)]$\set{e\in \SX}{\intf{e,e}=0, \;\intf{e, h}=2}$ is empty, and
\item[(e)]$\set{r\in \SX}{\intf{r,r}=-2, \;\intf{r, h}=0}$ is empty.
\end{itemize}
If $h\in \SX$ is a very ample polarization of degree $4$,
then the set $\FFF_h$ of classes of lines contained in the image $X_h$ of $\Phi_{h}\colon X\to \P^3$ is equal to
$$
\set{r\in \SX}{\intf{r,r}=-2, \;\intf{r, h}=1}.
$$
\end{theorem}
\begin{proof}
The condition (a) is equivalent to the condition  that $h$ is in the positive cone 
$\PPP(\SX)$ of $\SX\tensor \R$ containing $\hFer$.
Suppose that (a) holds.
Since the nef-cone of $X$ is a standard fundamental domain of 
the action of $W(\SX)$ on $\PPP(\SX)$,  
the condition (b) is equivalent to  the condition that $h$ is nef.
Suppose that (a) and (b) hold.
By Proposition 0.1 of~\cite{MR1260944}, 
the condition (c) is equivalent to  the condition that $|\LLL_h|$ is fixed-component free,
and hence defines a morphism $\Phi_h\colon X\to \P^3$.
Suppose that (a)--(c) hold.
By~\cite{MR0364263},
the condition (d) is equivalent to  the condition that $\Phi_{h}$ is not hyperelliptic, that is,
$\Phi_h$ is generically injective.
Suppose that (a)--(d) hold.
The condition (e) is equivalent to  the condition that $\Phi_{h}$ does not contract any $(-2)$-curves, that is,
the image $X_h$ of $\Phi_h$ is smooth.
The second assertion is obvious.
\end{proof}
Note that the conditions (a)--(e) can be checked by means of Algorithms~\ref{algo:AffES}~and~\ref{algo:Sep}, and 
 that the set $\FFF_h$ can be  calculated by Algorithm~\ref{algo:AffES}.
\par
We say that $h\in \SX$  is an \emph{$\Xfs$-polarization}
if $h$ is a very ample polarization  of degree $4$ such that $X_h\subset \P^3$ contains 
 exactly $56$ lines.
 The \emph{relative degree} of  a very ample polarization  $h$  of degree $4$ is defined to be $\intf{h, \hFer}$.
 \par
Using Algorithm~\ref{algo:AffES}, we calculate the set
 $$
 \HHH_d:=\set{v\in \SX}{\intf{v, \hFer}=d, \;\;\intf{v, v}=4}
 $$
 for $d=1, \dots, 6$.
 Note that $G_{48}$ acts on each  $\HHH_d$.
 We have $\HHH_d=\emptyset$ for $d<4$, and 
 $$
 \HHH_4=\{\hFer\},\;\;\;
 |\HHH_5|=48,\;\;\;
| \HHH_6|=48264.
 $$
 The action of $G_{48}$ on $\HHH_5$ is transitive,  and 
no vectors in $\HHH_5$ are nef.
 The action of $G_{48}$ decomposes $\HHH_6$ into $60$ orbits.
Among the vectors in $\HHH_6$,
\begin{itemize}
\item $792$ vectors in $5$ orbits are not nef,
\item $792$ vectors in $5$ orbits are nef, fixed-component free, but define hyperelliptic morphism,
\item $46296$ vectors in $48$ orbits are nef, fixed-component free,  define non-hyperelliptic morphism, but the images are singular, and 
\item the remaining $384$ vectors in $2$ orbits are very ample, and the images contain exactly $56$ lines.
The larger group $\tilGFer$ acts on these $384$ vectors transitively.
\end{itemize}
Thus we obtain the following theorem.
 \begin{theorem}
 {\rm (1)}
 If $h\in \SX$ is a very ample polarization of degree $4$
 with relative degree $6$, then $h$ is an $\Xfs$-polarization.
 {\rm (2)}
 There exist exactly $384$ $\Xfs$-polarizations of relative degree $6$.
Up to the action of $\Aut(\XFer)$ and $\Gal(\Q(\zeta)/\Q)$,
there exists only one $\Xfs$-polarization of relative degree $6$.  
\end{theorem}
We describe  $\Xfs$-polarizations of relative degree $6$ geometrically.
\begin{definition}
An ordered list   $(\ell_1, \ell_2, m_1, m_2, m_3, m_4, n)$ of seven lines on $\XFer$ is called an \emph{$\Xfs$-configuration} if the following conditions are satisfied:
\begin{itemize}
\item $\{\ell_1, \ell_2\}\in o_4$, 
\item $\{\ell_1, m_1\} \in o_1$, $\{\ell_2, m_1\} \in o_1$, and  $\{\ell_1, m_i\} \in o_3$, $\{\ell_2, m_i\} \in o_3$ for $i=2,3,4$, 
 \item $\{m_{1}, m_{k}\} \in o_7$ for $k=2,3,4$,
 \item $\{m_{2}, m_{3}\} \in o_5$, $\{m_{2}, m_{4}\} \in o_8$, $\{m_{3}, m_{4}\} \in o_8$, and 
 \item 
 $\{\ell_1, n\}\in o_8$,  $\{\ell_2, n\}\in o_8$,  $\{m_{1}, n\}\in o_8$,  $\{m_{2}, n\}\in o_2$,  $\{m_{3}, n\}\in o_2$,  $\{m_{4}, n\}\in o_7$.
  \end{itemize}
 We make the list of $\Xfs$-configurations.
 It turns out that the number of $\Xfs$-configurations is $6144$. 
 Comparing this list with the list of $\Xfs$-polarizations  of relative degree $6$,
 we obtain the following theorem.
\end{definition}
  \begin{theorem}
 If  $(\ell_1, \ell_2, m_{1}, \dots, m_{4}, n)$ is an $\Xfs$-configuration,
 then the vector 
\begin{equation}\label{eq:h3hX56conf}
 h:=3\hFer- ([\ell_1]+ [\ell_2]+ [m_{1}]+ \cdots+[m_{4}])
 \end{equation}
 is an $\Xfs$-polarization  of relative degree $6$.
 \end{theorem}
Consider the  seven lines $\ell_1, \ell_2, m_{1}, \dots, m_{4}, n$ labelled by the tags
\begin{equation}\label{eq:seveneqs}
[ 2, [ 1, 1 ] ], \;\; [ 2, [ 5, 5 ] ], \;\; [ 2, [ 1, 5 ] ], \;\; [ 3, [ 1, 1 ] ], \;\; [ 3, [ 3, 3 ] ], \;\; [ 4, [ 1, 7 ] ] ,\;\; [3, [1, 3]],
\end{equation}
respectively.
These lines form an $\Xfs$-configuration, and the corresponding $\Xfs$-polarization $\hfs$ is given by
$$
\hfs=[ 1, 2, 1, 2, 0, 0, 0, 0, 0, -1, -1, 0, 0, 0, -1, 0, 1, 1, 1, 0 ].
$$
\begin{theorem}\label{thm:X56defeq}
Let $\Phi_{56}\colon  \XFer\inj  \P^3$ be the embedding induced by  $\hfs$,
and $\Xfs$ the image of  $\Phi_{56}$.
With a suitable choice of the homogeneous coordinates of   $\P^3$, 
the surface $\Xfs$  is defined by the equation
$\Psi=0$, where $\Psi$ is given in Theorem~\ref{thm:main}.
\end{theorem}
%
%
\begin{proof}
Let $\Gamma_d$ be  the space of all homogeneous polynomials of degree $d$ in the variables $x_1, x_2, x_3, x_4$.
We have a natural identification  
$$
 \Gamma_3=H^0(\XFer, \LLL_{3 \hFer}).
 $$
 Since $\LLL_{\hfs}$ is isomorphic to $\LLL_{3 \hFer}(-\ell_1-\ell_2-m_{1}- \dots- m_{4})$ as an invertible sheaf, 
the space $H^0(\XFer, \LLL_{\hfs})$ is identified with the space of
homogeneous polynomials of degree $3$ in $x_1, x_2, x_3, x_4$ that vanish along 
each of the lines $\ell_1, \ell_2, m_{1}, \dots, m_{4}$.
Since we have explicit defining equations~\eqref{eq:seveneqs} of these lines,
we can confirm that $H^0(\XFer, \LLL_{\hfs})$ is of dimension $4$, and 
calculate a basis $f_1, f_2, f_3, f_4$ of $H^0(\XFer, \LLL_{\hfs})$ by elementary linear algebra.
 We fix a basis of $H^0(\XFer, \LLL_{\hfs})$ as in Table~\ref{table:fs}.
 %
 \begin{table}
 {\small
\begin{eqnarray*}
 f_1 &=& \left( 1+\zeta-{\zeta}^{3} \right) {x_{{1}}}^{3}+ \left( \zeta+{\zeta}^{2}+{\zeta}^{3} \right) {x_{{1}}}^{2}x_{{3}}+ \left( 1+\zeta \right) {x_{{1}}}^{2}x_{{4}}+ \left( -\zeta-{\zeta}^{2}-{\zeta}^{3} \right) x_{{1}}{x_{{2}}}^{2}+ \\
&&\left( -1-\zeta \right) x_{{1}}x_{{2}}x_{{3}}+ \left( \zeta+{\zeta}^{2} \right) x_{{1}}x_{{2}}x_{{4}}-x_{{1}}{x_{{3}}}^{2}+ \left( \zeta+{\zeta}^{2} \right) x_{{1}}x_{{3}}x_{{4}}-{\zeta}^{3}x_{{1}}{x_{{4}}}^{2}+\\
&& \left( 1-{\zeta}^{2}-{\zeta}^{3}\right) {x_{{2}}}^{2}x_{{3}}+ \left( -\zeta-{\zeta}^{2} \right) x_{{2}}{x_{{3}}}^{2}+ \left( {\zeta}^{2}+{\zeta}^{3} \right) x_{{2}}x_{{3}}x_{{4}}+{\zeta}^{2}{x_{{3}}}^{3}+x_{{3}}{x_{{4}}}^{2}\end{eqnarray*}
\begin{eqnarray*}
 f_2 &=& {x_{{1}}}^{3}-\zeta^2{x_{{1}}}^{2}x_{{3}}+ \left( -1+{\zeta}^{3} \right) {x_{{1}}}^{2}x_{{4}}-\zeta^2 x_{{1}}{x_{{2}}}^{2}+ \left( 1-{\zeta}^{3} \right) x_{{1}}x_{{2}}x_{{3}}+ \left( -1-\zeta \right) x_{{1}}x_{{2}}x_{{4}}+\\
 &&\left( 1+\zeta-{\zeta}^{3} \right) x_{{1}}{x_{{3}}}^{2}+ \left( -{\zeta}^{2}-{\zeta}^{3} \right) x_{{1}}x_{{3}}x_{{4}}+ \left( -1-\zeta-{\zeta}^{2} \right) x_{{1}}{x_{{4}}}^{2}+\zeta\,{x_{{2}}}^{2}x_{{3}}+\\
 &&\left( {\zeta}^{2}+{\zeta}^{3} \right) x_{{2}}{x_{{3}}}^{2}+ \left( 1-{\zeta}^{3} \right) x_{{2}}x_{{3}}x_{{4}}+ \left( \zeta+{\zeta}^{2}+{\zeta}^{3} \right) {x_{{3}}}^{3}+ \left( 1+\zeta-{\zeta}^{3} \right) x_{{3}}{x_{{4}}}^{2}
\end{eqnarray*}
\begin{eqnarray*}
  f_3 &=&\left( 1+\zeta+{\zeta}^{2} \right) {x_{{1}}}^{2}x_{{2}}+ \left( \zeta+{\zeta}^{2}+{\zeta}^{3} \right) {x_{{1}}}^{2}x_{{4}}+ \left( -1-\zeta\right) x_{{1}}x_{{2}}x_{{3}}+ \left( \zeta+{\zeta}^{2} \right) x_{{1}}x_{{2}}x_{{4}}+ \\
&&\left( -\zeta-{\zeta}^{2} \right) x_{{1}}x_{{3}}x_{{4}}+ \left( {\zeta}^{2}+{\zeta}^{3} \right) x_{{1}}{x_{{4}}}^{2}+ \left( 1-{\zeta}^{2}-{\zeta}^{3} \right) {x_{{2}}}^{3}+ \left( -\zeta-{\zeta}^{2} \right) {x_{{2}}}^{2}x_{{3}}+\\
&&\left( 1+\zeta+{\zeta}^{2} \right) {x_{{2}}}^{2}x_{{4}}+\zeta^2 x_{{2}}{x_{{3}}}^{2}+ \left( -{\zeta}^{2}-{\zeta}^{3} \right) x_{{2}}x_{{3}}x_{{4}}+{\zeta}^{3}x_{{2}}{x_{{4}}}^{2}+{\zeta}^{3}{x_{{3}}}^{2}x_{{4}}+\zeta\,{x_{{4}}}^{3}
  \end{eqnarray*}
\begin{eqnarray*}
   f_4 &=&-\zeta\,{x_{{1}}}^{2}x_{{2}}+{x_{{1}}}^{2}x_{{4}}+ \left( -1+{\zeta}^{3} \right) x_{{1}}x_{{2}}x_{{3}}+ \left( 1+\zeta \right) x_{{1}}x_{{2}}x_{{4}}+ \left( -{\zeta}^{2}-{\zeta}^{3} \right) x_{{1}}x_{{3}}x_{{4}}+ \\
&&\left( -1+{\zeta}^{3} \right) x_{{1}}{x_{{4}}}^{2}+{\zeta}^{3}{x_{{2}}}^{3}+ \left( -1-\zeta \right) {x_{{2}}}^{2}x_{{3}}+\zeta\,{x_{{2}}}^{2}x_{{4}}+ \left( -1-\zeta+{\zeta}^{3} \right) x_{{2}}{x_{{3}}}^{2}+\\
&&\left( 1-{\zeta}^{3} \right) x_{{2}}x_{{3}}x_{{4}}+ \left( -1+{\zeta}^{2}+{\zeta}^{3} \right) x_{{2}}{x_{{4}}}^{2}+ \left( 1-{\zeta}^{2}-{\zeta}^{3} \right) {x_{{3}}}^{2}x_{{4}}+ \left( -1-\zeta-{\zeta}^{2} \right) {x_{{4}}}^{3}
   \end{eqnarray*}
}
\caption{Basis of $H^0(\XFer, \LLL_{\hfs})$}\label{table:fs}
\end{table}
 Let $\bar\Gamma_{12}$ denote  the space of all homogeneous polynomials of degree $12$ in $x_1, x_2, x_3, x_4$
 such that the degree with respect to $x_1$ is $\le 3$.
 For $g\in \Gamma_{12}$,
 let $\rho(g)$ denote the remainder on 
the division by $x_1^4+x_2^4+x_3^4+x_4^4$ under the lex monomial ordering $x_1>x_2>x_3>x_4$.
Then we obtain  a surjective homomorphism 
 $$
\rho\colon   \Gamma_{12} \to \bar\Gamma_{12}.
 $$
 Therefore we have a natural identification  
$$
\bar\Gamma_{12}=H^0(\XFer, \LLL_{12 \hFer}).
 $$
 Let $\Sigma_4$ be the linear space of all homogeneous polynomials 
 of degree $4$ in the variables $y_1, y_2, y_3, y_4$.
 The substitution $y_i\mapsto f_i$ for $i=1, \dots, 4$
 gives rise to a linear homomorphism $\sigma\colon \Sigma_4\to \Gamma_{12}$.
 The linear homomorphism $\rho\circ\sigma$ is represented by a $290\times 35$ matrix.
 The kernel of $\rho\circ\sigma\colon  \Sigma_4\to \bar\Gamma_{12}$
 is of dimension  $1$ and is generated by the polynomial $\Psi\in \Sigma_4$.
   \end{proof}
We study the projective geometry of the surface $\Xfs$ more closely.
The set 
$$
\Lines_{56}:=\set{r\in \SX}{\intf{r, r}=-2, \;\; \intf{r, \hfs}=1}
$$
of classes of lines on $\Xfs$ can be easily calculated by Algorithm~\ref{algo:AffES}.
It turns out that $\Lines_{56}$ spans $\SX$, and hence
 the stabilizer subgroup 
$$
\tilGfs:=\set{g\in \OG(\SX)}{\hfs^g=\hfs}
$$
of $\hfs$ in  $ \OG(\SX)$ is naturally isomorphic to the group of permutations of $\Lines_{56}$ that preserve the intersection numbers.
We fix a list of  vectors $[\lambda_1], \dots, [\lambda_{20}]$ of $\FFF_{56}$ that form 
a basis of $\SX$,
and calculate $\tilGfs$ by the  standard backtrack program in the same way as the calculation of $\tilGFer$.
It turns out that $\tilGfs$  is  of order $128$.
We put
$$
\Aut(\Xfs):=\set{\gamma\in \PGL_4(\C)}{ \gamma(\Xfs)=\Xfs},
$$
and consider its image 
$$
G_{56}:=\Image(\Aut(\Xfs) \to \OG(\SX))
$$
by the natural representation.
Then $G_{56}$ is a subgroup of   $\tilGfs$ consisting of elements that satisfy the period-preserving condition~\eqref{eq:preserveomegacond}.
It turns out that $G_{56}$ is isomorphic to the group 
$$
(((\Z/4\Z \times \Z/2\Z) : \Z/2\Z) : \Z/2\Z) : \Z/2\Z
$$
of order $64$.
The action of  $\Aut(\Xfs)$ decomposes $\Lines_{56}$ into three orbits $O_8$, $O_{16}$, $O_{32}$ of size $8$, $16$, and $32$,
respectively.
We have
\begin{equation}\label{eq:h56lincomboflines}
\hfs=\frac{1}{8} \sum_{r\in O_{32}} r=\frac{1}{8} \left(2 \sum_{r\in O_{8}} r + \sum_{r\in O_{16}} r\right).
\end{equation}
The intersection $\Lines_{48} \cap \Lines_{56}$ consists of $30$ classes.
(In fact, in searching for a basis of $H^0(\XFer, \LLL_{\hfs})$ 
that gives a simple  defining equation of $\Xfs$, 
we have used  these $30$ lines as a clue.)
Since we know the defining equations of the $48$ lines on $\XFer$ and the morphism $\Phi_{56}\colon \XFer\isom \Xfs$
explicitly, 
we can easily compute 
the defining equations of these $30$ lines.
\par
We say that a finite set $\{\mu_1, \dots, \mu_N\}$ of lines in $\P^3$
\emph{has a unique common intersecting line}
if there exists a unique line in $\P^3$ that intersects all of $\mu_1, \dots, \mu_N$.
If we know the defining equations of $\mu_1, \dots, \mu_N$,
then we can determine whether 
$\{\mu_1, \dots, \mu_N\}$ has a unique common intersecting line,
and if it has,
we can calculate the defining equation of the common intersecting line.
Suppose that we know the defining equations of 
lines $\lambda\sprime_1, \dots, \lambda\sprime_N$ on $\Xfs$,
but we do not know the defining equation of a line $\lambda\sprime_{N+1}$ on $\Xfs$.
Since we have the set $\FFF_{56}$ of classes of lines on $\Xfs$,
we can make the subset $\{\lambda\sprime_{i_1}, \dots, \lambda\sprime_{i_k}\}$ of
$\{\lambda\sprime_1, \dots, \lambda\sprime_N\}$
consisting of lines that intersect $\lambda\sprime$.
If $\{\lambda\sprime_{i_1}, \dots, \lambda\sprime_{i_k}\}$ has a unique common intersecting line,
then we can calculate the defining equation of $\lambda\sprime_{N+1}$.
Starting from the $30$ lines,
we can calculate the defining equations of the remaining $26$ lines on $\Xfs$.
Since we know the permutation action of $G_{56}$ on $\Lines_{56}$,
we can calculate each element of $\Aut(\Xfs)$.
By these calculations, we obtain the following theorem.
\begin{theorem}
The subgroup $\Aut(\Xfs)$ of $\PGL_4(\C)$ is generated by the following elements of order $4$.
\begin{eqnarray*}
\gamma_1&=&\left[ \begin {array}{cccc} 1&{\zeta}^{2}&-\zeta+{\zeta}^{2}-{\zeta}^
{3}&1-\zeta+{\zeta}^{3}\\ \noalign{\medskip}-1+\zeta-{\zeta}^{3}&-
\zeta+{\zeta}^{2}-{\zeta}^{3}&-{\zeta}^{2}&1\\ \noalign{\medskip}-1&{
\zeta}^{2}&-\zeta+{\zeta}^{2}-{\zeta}^{3}&-1+\zeta-{\zeta}^{3}
\\ \noalign{\medskip}-1+\zeta-{\zeta}^{3}&\zeta-{\zeta}^{2}+{\zeta}^{3
}&{\zeta}^{2}&1\end {array} \right],\\
\gamma_2&=&\left[ \begin {array}{cccc} 1&{\zeta}^{2}&\zeta+{\zeta}^{2}-{\zeta}^{
3}&1-\zeta-{\zeta}^{3}\\ \noalign{\medskip}{\zeta}^{2}&1&1-\zeta-{
\zeta}^{3}&\zeta+{\zeta}^{2}-{\zeta}^{3}\\ \noalign{\medskip}\zeta+{
\zeta}^{2}-{\zeta}^{3}&1-\zeta-{\zeta}^{3}&-1&-{\zeta}^{2}
\\ \noalign{\medskip}1-\zeta-{\zeta}^{3}&\zeta+{\zeta}^{2}-{\zeta}^{3}
&-{\zeta}^{2}&-1\end {array} \right].
\end{eqnarray*}
The lines on $\Xfs$ are obtained from the following lines $\lambda\spar{8}, \lambda\spar{16}, \lambda\spar{32}$ by the action of  $\Aut(\Xfs)$,
where the size of the orbit of $\lambda\spar{n}$ are $n$.
\begin{eqnarray*}
\lambda\spar{8} &:& y_1+\zeta^2\,y_4=y_2-\zeta^2\,y_3=0, \\
\lambda\spar{16} &:& y_1+(-\zeta+\zeta^2-\zeta^3)\, y_4=y_2+(-\zeta+\zeta^2-\zeta^3) \,y_3=0, \\
\lambda\spar{32} &:& y_1=3 y_2+ (-1-\zeta-\zeta^3)  y_3+(-\zeta+\zeta^2+\zeta^3)\,y_4=0.
\end{eqnarray*}
\end{theorem}
\begin{corollary}\label{cor:outside3}
Every line $\lambda$ on $\Xfs$  is defined by an equation $M_{\lambda}\,y=0$,
where $M_{\lambda}$ is a $2\times 4$ matrix in the row-reduced echelon form with components in $\Z[\zeta, 1/3]$,
and $y=[y_1, y_2, y_3, y_4]^{T}$.
\end{corollary}

%
 %
%
\section{Reductions of $\Xfs$ at primes of $\Z[\zeta]$}\label{sec:reduction}
\subsection{Buchberger algorithm and the reduction}
We need a slight enhancement of the Buchberger algorithm
to calculate Gr\"obner bases over $\kappa_P$ at all but finitely many  primes $P$ of $\Z[\zeta]$ simultaneously.
This method must have been used by many people without fanfare,
but we cannot find any appropriate references.
\par
We fix a monomial ordering on the set of monomials of variables $z_1, \dots, z_n$.
Let $F$ be a field.
We use the notation in Chapter 2 of~\cite{MR1417938}.
In particular, 
for a non-zero polynomial $f\in F[z_1, \dots, z_n]$,
let $\LC(f)\in F$ denote the leading coefficient,
and 
for $f, g\in F[z_1, \dots, z_n]$ and a  finite subset $H\subset F[z_1, \dots, z_n]$,
let $S(f, g)$ be the $S$-polynomial of $f$ and $g$,
and $\divbyH{f}{H}$ the remainder on the division  of $f$ by $H$.
\par
Suppose that $F$ is a number field, and 
let $R$ be the integer ring of $F$.
For a prime $P$ of $R$, 
let $R_P$ denote the localization of $R$ at $P$,
$R_P\sptimes$ the group of units of $R_P$,
and $\kappa_P$ the residue field of $R$ at $P$.
For a polynomial $f\in R_P[z_1, \dots, z_n]$, 
let $f\modP \in \kappa_P[z_1, \dots, z_n]$ denote the reduction of $f$ at $P$,
and for a subset $H$ of $R_P[z_1, \dots, z_n]$,
let $H\modP $ denote the set of reductions at $P$ of polynomials in $H$.
The following lemma follows immediately from the definition of 
the $S$-polynomial and the division algorithm.
\begin{lemma}\label{lem:SH}
Let $f$ and $g$ be polynomials in $R_P[z_1, \dots, z_n]$,
and $H$ a finite subset of $R_P[z_1, \dots, z_n] $.
We have $S(f, g)\in F[z_1, \dots, z_n]$  and $\divbyH{f}{H} \in  F[z_1, \dots, z_n]$.
\par
{\rm (1)} Suppose that   both of $\LC(f) $ and $\LC(g)$ belong to $R_P\sptimes$.
Then 
$S(f, g)$ belongs to $R_P[z_1, \dots, z_n]$,  and $S(f, g) \modP $ is equal to
the $S$-polynomial 
of the polynomials  $f\modP $ and $ g\modP $ in $ \kappa_P[z_1, \dots, z_n]$.
\par
{\rm (2)}
Suppose that  $\LC(h) \in R_P\sptimes$  for any $h\in H$.
Then 
$\divbyH{f}{H}$ belongs to $R_P[z_1, \dots, z_n]$,  and $\divbyH{f}{H} \modP $ is equal to
 the remainder on the division  of $f \modP  \in  \kappa_P[z_1, \dots, z_n]$ by  the subset  $H\modP $ of $ \kappa_P[z_1, \dots, z_n]$.
\end{lemma}
Suppose that a finite set $\{f_1, \dots, f_s\}$ of non-zero polynomials in $R[z_1, \dots, z_n]$
is given.
Let $I_F$ be the ideal of $F[z_1, \dots, z_n]$ generated by $\{f_1, \dots, f_s\}$.
For a prime  $P$ of $R$,
let $I_{P}$ be the ideal of $\kappa_P[z_1, \dots, z_n]$ generated by $\{f_1, \dots, f_s\}\modP $.
A Gr\"obner basis $G$ of  $I_F$  is calculated by the Buchberger algorithm.
We initialize $G:=\{f_1, \dots, f_s\}$.
If  $\divbyH{S(f_i, f_j)}{G}$  is non-zero for a pair of $f_i, f_j\in G$, 
we add $ f_t:=\divbyH{S(f_i, f_j)}{G}$ to $G$.
We continue this process until 
 no new non-zero polynomials $\divbyH{S(f_i, f_j)}{G}$ are obtained.
%
\par
We introduce a variable set $C$ in the Buchberger algorithm.
We initialize 
$$
C:=\set{\LC(f_i)}{ i=1, \dots, s},
$$
and, whenever a new non-zero polynomial $f_t=\divbyH{S(f_i, f_j)}{G}$ is added to $G$,
we add $\LC(f_t)$ to $C$.
From Lemma~\ref{lem:SH}, we obtain the following proposition.
\begin{proposition}
Let $P$ be a prime of $R$.
Suppose that,  when the algorithm terminates, 
we have $C\subset R_P\sptimes$.
Then we have $G\subset R_P[z_1, \dots, z_n]$, and $G\modP $ is a Gr\"obner basis of the ideal $I_{P}$ of $\kappa_P[z_1, \dots, z_n]$.
\end{proposition}
Since $C$ is a finite set,
we can calculate a finite set $S$ of  prime integers  such that $G\modP $ is  a Gr\"obner basis of $I_{P}$ for any  prime $P$ over $p\notin S$.
More precisely,
for $\alpha\in F$,
let $d(\alpha)$ denote the least positive integer such that $d(\alpha)\alpha \in R$,
and let $n(\alpha)\in \Z$ be the norm of $d(\alpha)\alpha\in R$ over $\Z$.
Let $\tilde{C}$
denote the subset
\begin{equation}\label{eq:nds}
\shortset{d(\alpha)}{\alpha\in C}\cup\shortset{n(\alpha)}{\alpha\in C}
\end{equation}
of $\Z\setminus\{0\}$.
For a finite set $T$ of non-zero integers,
let $\PPP(T)$ denote the set of prime integers that divide at least one element of $T$.
Then $C\subset R_P\sptimes$ holds  for any prime $P$ over $p\notin \PPP(\tilde{C})$.
\par
In fact, 
this naive method often fails to  work in practice,
because some elements of $\tilde{C}$ can be so large
that we cannot calculate their prime factors.
(For example,  in the proof of Theorem~\ref{thm:redX56} below,
this method led us to a factorization of  a composite  integer $>10^{80}$,
which was impossible.)
To overcome this difficulty, 
we use the following trick. 
Let $T_1, \dots, T_N$ be finite  sets of non-zero integers.
Then we have
$$
\PPP(T_1)\cap\dots \cap  \PPP(T_N)=\PPP(\gcds(T_1, \dots, T_N)),
$$
where
$$
\gcds(T_1, \dots, T_N):=\set{\gcd(t_1, \dots, t_N)}{t_1\in T_1, \dots,  t_N\in T_N}.
$$
Since the calculation of the greatest common divisor  of large  integers is much easier than the calculation of prime factors
of these integers,
we often  manage  to  calculate $\PPP(T_1)\cap\dots \cap  \PPP(T_N)$
even when the calculation of $\PPP(T_i)$ is intractable.
\par
For example, 
suppose that $I_F$ contains $1$, and 
let us calculate a finite set $S$ of prime integers such that 
$1\in I_P$  holds for any prime $P$ of $R$ over $p\notin S$.
We carry out  the Buchberger algorithm 
several times under various choices of monomial ordering,
and obtain the sets $\tilde{C}_1, \dots, \tilde{C}_N$ of non-zero integers for these choices. 
Note that,  if $p\notin \PPP(\tilde{C}_{\nu})$  for at least one  $\nu$,
then $i \in I_P$ for any prime $P$ of $R$ over $p$.
Hence the intersection $S$ of these $\PPP(\tilde{C}_i)$ has the desired property.
\par
By means of  this method,
we write the following algorithms.
\begin{algorithm}\label{algo:reductionsmooth}
Let $V$ be a subscheme of  $\P^3$ defined by a homogeneous polynomial
$\psi\in R[z_1, \dots, z_4]$ such that 
 $V\tensor F$ is a smooth surface.
Then we can make a finite set $S$
of prime integers such that $V\tensor\kappa_P$
is  a smooth surface for any prime $P$ of $R$ over $p\notin S$.
Executing the Buchberger algorithm in the field $\kappa_P$ 
for the primes $P$ over $p\in S$,
we can make the complete set of primes $P$ 
such that $V\tensor \kappa_P$ is not a smooth surface.
\end{algorithm}
\begin{algorithm}\label{algo:reductionlines}
We say that a finite set $\{\mu_1, \dots, \mu_N\}$
of lines in $\P^3$ defined over a field  \emph{has no common intersecting lines}
if there exist no lines in $\P^3$ that intersect all of $\mu_1, \dots, \mu_N$.
Let  $\{\mu_1, \dots, \mu_N\}$ be a set of subschemes of $\P^3$  defined over $R$
such  that  $\{\mu_1\tensor F, \dots, \mu_N\tensor F\}$  is a set of distinct $N$ lines with  
no common intersecting lines.
We can make a complete set $B$ of primes of $R$
such that $\{\mu_1\tensor \kappa_P, \dots, \mu_N \tensor \kappa_P\}$ 
is a set of distinct $N$ lines that has  no  common intersecting lines for any $P\not\in B$.
\end{algorithm}
\subsection{Reductions of $\Xfs$}
Let $\XXXfs$ be the projective scheme over $\Z[\zeta]$ defined by the homogeneous equation $\Psi=0$ in $\P^3$.
The generic fiber $\XXXfs\tensor \Q(\zeta)$ is the surface $\Xfs$.
For a prime $P$ of $\Z[\zeta]$, 
let 
$\bar{\kappa}_P$ denote an algebraic closure of $\kappa_P$.
We  define  $\XfsP$ to be the pullback of $\XXXfs$ by $ \Z[\zeta]\to \bar\kappa_P$.
\par
Recall that $A= -1-2\zeta-2\zeta^3$ and $B=3+A$.
There exists only one  prime $P_2$ of $\Z[\zeta]$ over $2$,
and 
we have $A \bmod{P_2}=1$.
There exist exactly two primes $P_3$ and $P\sprime_3$ of $\Z[\zeta]$ over $3$,
for which  we have $A \bmod{P_3}=0$ and $A \bmod{P\sprime_3}=1$.
It is easy to see that 
 $\Xfs\ppar{P_2}$ and $\Xfs\ppar{P\sprime_3}$ are singular  at the point $(1:0:\sqrt{-1}:0)$.
%
%

%
\begin{proposition}\label{prop:char3}
The surface $\Xfs\ppar{P_3}$ is projectively isomorphic  over $\kappa_{P_3}\cong \F_9$ to
the Fermat quartic surface in characteristic $3$.
In particular, $\Xfs\ppar{P_3}$ is smooth,  and contains $112$ lines,
each of which is defined over $\kappa_{P_3}$.
\end{proposition}
\begin{proof}
The surface $\Xfs\ppar{P_3}$ is defined by 
$y_1^3 y_2+y_1 y_2^3+y_3^3 y_4+y_3 y_4^3=0$,
which is a non-degenerate  Hermitian form in $4$ variables over $\F_{9}$.
Hence $\Xfs\ppar{P_3}$ is projectively isomorphic to
the Fermat quartic surface in characteristic $3$ over $\F_9$
by~\cite{MR0213949}.
For the number of lines on this surface, see~\cite{MR0213949} or~\cite{MR3190354}.
\end{proof}
By Corollary~\ref{cor:outside3},
the lines on $\Xfs$ reduce to lines on $\XfsP$  for any prime $P$ over $p>3$.
\begin{theorem}\label{thm:redX56}
Suppose that $P$ is a prime of $\Z[\zeta]$ over a prime integer  $p>3$.
Then  $\XfsP$  is smooth,
and contains  exactly $56$ lines,
each of which is obtained by  the reduction of a line on $\Xfs$ at $P$.
Moreover, the isomorphism $\XFer\isom \Xfs$ given in Table~\ref{table:fs}
reduces at $P$ to an isomorphism  $\XFer\tensor\bar\kappa_P \isom \XfsP$.
\end{theorem}
\begin{proof}
The smoothness of  $\XfsP$ can be proved by Algorithm~\ref{algo:reductionsmooth}.
We show that  $\XfsP$ contains exactly $56$ lines,
and that they are obtained by the reduction of lines on $\Xfs$.
The fact that
the $56$ lines on $\Xfs$ reduce to distinct $56$ lines on $\XfsP$
keeping the intersection numbers
can be easily proved.
We will show that there exist no other lines on $\XfsP$.
Let $\SXp$ denote the N\'eron-Severi lattice of $\XfsP$.
Recall that  the set $\FFF_{56}$ of classes of the $56$ lines on $\Xfs$ 
spans $\SX$.
Hence the reduction of lines on $\Xfs$  
induces a natural embedding
$$
\SX\inj \SXp
$$
of lattices. 
From now on,
we regard $\SX$ as a sublattice of $\SXp$ by this embedding.
In particular, $\FFF_{56}$ is a subset of  the set of classes of lines on $\XfsP$.
It is enough to show that, if $\lambda$ is a line on $\XfsP$,
then its class $[\lambda]\in \SXp$ is in  $\FFF_{56}$.
Let $Q_P$ denote the orthogonal complement of $\SX$ in $\SXp$.
Then $Q_P$ is either of rank $0$ or negative-definite of rank $2$.
(See~\cite{MR2452829} for the problem when $Q_P$ is of rank $2$.)
We have
 $$
 \SX\oplus Q_P\;\;\subset\;\; \SXp\;\; \subset\;\; \SX\dual\oplus Q_P\dual.
 $$
We denote the projections by 
$$
\pr_{S}\colon \SXp\to \SX\dual,
\qquad
\pr_Q\colon \SXp\to Q_P\dual.
$$
Let $\hfs(P)\in \SXp$  denote the class   of a hyperplane section of $\XfsP\subset \P^3$.
Since $\FFF_{56}$ is perpendicular to $Q_P\dual$,
we have 
$\intf{\pr_S(\hfs(P)), r}=\intf{\hfs(P), r}=1$ for any $r\in \FFF_{56}$.
The class $\hfs\in \SX$ is characterized in $\SX\tensor\Q$ by the property that 
$\intf{\hfs, r}=1$ holds for any $r\in \FFF_{56}$.
Hence we have $\pr_S(\hfs(P))=\hfs$.
Since $\intf{\hfs(P), \hfs(P)}=\intf{\hfs, \hfs}=4$ and $Q_P$ is either $0$ or negative-definite,
we obtain $\pr_Q(\hfs(P))=0$, and therefore we have
$$
\hfs(P)=\hfs.
$$
Suppose that there exists a line  $\nu$  on $\XfsP$
such that $[\nu]\notin \FFF_{56}$.
Since $\hfs(P)=\hfs$ and $\hfs\perp Q_P$,  we have $\intf{\pr_S([\nu]), \hfs}=1$.
Since $\intf{\pr_Q([\nu]), \pr_Q([\nu])}\le 0$,  we have $\intf{\pr_S([\nu]), \pr_S([\nu])}\ge -2$.
For any $[\lambda]\in \FFF_{56}$, we have $\intf{[\nu], [\lambda]}\in \{0, 1\}$.
Since $Q_P\perp \FFF_{56}$, 
we have $\intf{\pr_S([\nu]), [\lambda]}\in \{0, 1\}$
for any $[\lambda]\in \FFF_{56}$.
Therefore $\pr_S([\nu])$ belongs to 
$$
 \FFF\sprime_{56}:=\set{r\sprime \in \SX\dual}{\intf{r\sprime, \hfs}=1, \;\intf{r\sprime,r\sprime}\ge -2, \; \intf{r\sprime, [\lambda]}\in \{0, 1\}
 \;\textrm{for any $[\lambda]\in \FFF_{56}$}}.
$$
We calculate $ \FFF\sprime_{56}$ by Algorithm~\ref{algo:AffES}. 
It turns out that $ \FFF\sprime_{56}$ consists of $56$ vectors.
For each $r\sprime\in \FFF\sprime_{56}$,
we calculate the set $\Lambda(r\sprime)=\{\lambda\sprime_1, \dots, \lambda\sprime_k\}$
of lines on $\Xfs$ such that
$$
\{[\lambda\sprime_1], \dots, [\lambda\sprime_k]\}=\set{[\lambda]\in \FFF_{56}}{\intf{[\lambda], r\sprime}=1}.
$$
Then $\nu$ is a common intersecting line of the set
$$
\Lambda(\pr_S([\nu]))\tensor\kappa_P=\{\lambda\sprime_1\tensor\kappa_P, \dots, \lambda\sprime_k\tensor\kappa_P\}
$$
of lines over $\kappa_P$.
On the other hand, since we know the defining equations over $\Z[\zeta, 1/3]$ of lines $\lambda\sprime_1, \dots, \lambda\sprime_k$,
 we can see by Algorithm~\ref{algo:reductionlines}
 that $\Lambda(r\sprime)\tensor\kappa_P$ has no  common intersecting lines for any $r\sprime\in \FFF\sprime_{56}$ if $P$ is over $p>3$.
 Thus we obtain a contradiction.
  (See Remark~\ref{rem:redP3} for what happens when $p=3$.)
\par
Next,  we investigate the reduction at $P$ of the isomorphism $\Phi_{56}\colon \XFer\isom \Xfs$ given in Table~\ref{table:fs}.
The equality~\eqref{eq:h56lincomboflines} implies that,  as an invertible sheaf, the line bundle $\LLL_{\hfs}^{\tensor 8}$ is isomorphic to $\OOO(D)$,
where $D$ is a linear combination of lines on $\Xfs$.
Hence the embedding $\SX\inj \SXp$ induced by the reduction of lines on $\Xfs$ 
maps the class $\hfs$ of $\LLL_{\hfs}$ to the class of 
the line bundle $\LLL_{\hfs}\tensor\bar\kappa_P$.
Since  $\hfs=\hfs(P)$,
the line bundle $\LLL_{\hfs}\tensor\bar\kappa_P$ is the very ample line bundle associated with the embedding $\XfsP\inj \P^3$.
We confirm that 
$$
f\modP  := (f_1 \modP , \dots, f_4\modP )
$$
 are linearly independent over $\kappa_P$.
Hence $f\modP $  form a basis of the space of the global sections of 
 $\LLL_{\hfs}\tensor\bar\kappa_P$.
\end{proof}
\begin{remark}\label{rem:redP3}
We investigate the lines on $\Xfs (P_3)$,
where $P_3$ is the prime of $\Z[\zeta]$ 
in Proposition~\ref{prop:char3}.
For each line $\lambda$ on $\Xfs$, we have an invertible $2\times 2$ matrix $U_{\lambda}$ with components in $\Q(\zeta)$
such that $U_{\lambda} M_{\lambda}$  has components in the localization  $\Z[\zeta]_{P_3}$, 
and  that $\overline{M}\sprime_{\lambda}:=U_{\lambda} M_{\lambda} \bmod P_3$ is a matrix of row-reduced echelon form of rank  $2$.
Hence the  equation $\overline{M}\sprime_{\lambda}\,y=0$ defines a line on $\Xfs(P_3)$ defined over $\kappa_{P_3}$.
Using these equations, we can make the reduction  $\lambda\mapsto \lambda\tensor \kappa_{P_3}$ of  lines on $\Xfs$ to lines on $\Xfs(P_3)$.
Since this reduction keeps the intersection numbers,
it induces an embedding $\SX\inj S_{\Xfs(P_3)}$.
For each $r\sprime \in \FFF\sprime_{56}$,
the set $\Lambda(r\sprime)\tensor\kappa_{P_3}$ of lines over $\kappa_{P_3}$ has a unique common intersecting line,
and the common intersecting line is contained in $\Xfs(P_3)$.
Thus we obtain the $112=|\FFF_{56}|+|\FFF\sprime_{56}|$ lines on $\Xfs(P_3)$.
\end{remark}
\bibliographystyle{plain}

\begin{thebibliography}{10}

\bibitem{DIS}
A.~{Degtyarev}, I.~{Itenberg}, and A.~{Sinan Sert{\"o}z}.
\newblock {Lines on quartic surfaces}.
\newblock {\em ArXiv e-prints}, January 2016.


\bibitem{newDeg}
A.~{Degtyarev}.
\newblock {Smooth models of singular $K3$-surfaces}.
\newblock {\em ArXiv e-prints}, August 2016.


\bibitem{MR1228206}
Henri Cohen.
\newblock {\em A course in computational algebraic number theory}, volume 138
  of {\em Graduate Texts in Mathematics}.
\newblock Springer-Verlag, Berlin, 1993.

\bibitem{MR1417938}
David Cox, John Little, and Donal O'Shea.
\newblock {\em Ideals, varieties, and algorithms}.
\newblock Undergraduate Texts in Mathematics. Springer-Verlag, New York, second
  edition, 1997.
  algebra.

\bibitem{gap}
The~GAP Group.
\newblock {G}{A}{P} - {G}roups, {A}lgorithms, and {P}rogramming.
\newblock Version 4.7.9; 2015 (http://www.gap-system.org).

\bibitem{MR0373371}
Donald~E. Knuth.
\newblock Estimating the efficiency of backtrack programs.
\newblock {\em Math. Comp.}, 29:122--136, 1975.

\bibitem{MR3190354}
Shigeyuki Kond{\=o} and Ichiro Shimada.
\newblock The automorphism group of a supersingular {$K3$} surface with {A}rtin
  invariant 1 in characteristic 3.
\newblock {\em Int. Math. Res. Not. IMRN}, 2014(7):1885--1924, 2014.

\bibitem{MR682664}
A.~K. Lenstra, H.~W. Lenstra, Jr., and L.~Lov{\'a}sz.
\newblock Factoring polynomials with rational coefficients.
\newblock {\em Math. Ann.}, 261(4):515--534, 1982.


\bibitem{MR0168559}
Hideyuki Matsumura and Paul Monsky.
\newblock On the automorphisms of hypersurfaces.
\newblock {\em J. Math. Kyoto Univ.}, 3:347--361, 1963/1964.


\bibitem{Mizukami}
M.~Mizukami.
\newblock Birational mappings from quartic surfaces to {K}ummer surfaces.,
  1975.
\newblock Master's Thesis at University of Tokyo, in Japanese.


\bibitem{MR525944}
V.~V. Nikulin.
\newblock Integer symmetric bilinear forms and some of their geometric
  applications.
\newblock {\em Izv. Akad. Nauk SSSR Ser. Mat.}, 43(1):111--177, 238, 1979.
\newblock English translation: Math USSR-Izv. 14 (1979), no. 1, 103--167
  (1980).

\bibitem{MR1260944}
V.~V. Nikulin.
\newblock Weil linear systems on singular {$K3$} surfaces.
\newblock In {\em Algebraic geometry and analytic geometry (Tokyo, 1990)},
  ICM-90 Satell. Conf. Proc., pages 138--164. Springer, Tokyo, 1991.

\bibitem{Ogui}
K.~{Oguiso}.
\newblock {Isomorphic quartic K3 surfaces in the view of Cremona and projective
  transformations}.
\newblock {\em ArXiv e-prints}, February 2016.


\bibitem{MR0284440}
I.~I. Piatetski-Shapiro and I.~R. Shafarevich.
\newblock Torelli's theorem for algebraic surfaces of type {${\rm K}3$}.
\newblock {\em Izv. Akad. Nauk SSSR Ser. Mat.}, 35:530--572, 1971.
\newblock Reprinted in I. R. Shafarevich, Collected Mathematical Papers,
  Springer-Verlag, Berlin, 1989, pp.~516--557.


\bibitem{MR3343894}
S{\l}awomir Rams and Matthias Sch{\"u}tt.
\newblock 64 lines on smooth quartic surfaces.
\newblock {\em Math. Ann.}, 362(1-2):679--698, 2015.


\bibitem{MR3396099}
S{\l}awomir Rams and Matthias Sch{\"u}tt.
\newblock 112 lines on smooth quartic surfaces (characteristic 3).
\newblock {\em Q. J. Math.}, 66(3):941--951, 2015.



\bibitem{arXiv151201358}
S.~{Rams} and M.~{Sch{\"u}tt}.
\newblock {At most 64 lines on smooth quartic surfaces (characteristic 2)}.
\newblock {\em ArXiv e-prints}, December 2015.

\bibitem{MR0364263}
B.~Saint-Donat.
\newblock Projective models of {$K-3$} surfaces.
\newblock {\em Amer. J. Math.}, 96:602--639, 1974.

\bibitem{MR2653207}
Matthias Sch{\"u}tt, Tetsuji Shioda, and Ronald van Luijk.
\newblock Lines on {F}ermat surfaces.
\newblock {\em J. Number Theory}, 130(9):1939--1963, 2010.


\bibitem{MScomm}
Matthias Sch{\"u}tt.
\newblock {P}rivate communication, 2015.

\bibitem{MR0010431}
B.~Segre.
\newblock The maximum number of lines lying on a quartic surface.
\newblock {\em Quart. J. Math., Oxford Ser.}, 14:86--96, 1943.

\bibitem{MR0011003}
B.~Segre.
\newblock On the quartic surface {$x_1^4+x_2^4+x_3^4+x_4^4=0$}.
\newblock {\em Proc. Cambridge Philos. Soc.}, 40:121--145, 1944.

\bibitem{MR0213949}
B.~Segre.
\newblock Forme e geometrie hermitiane, con particolare riguardo al caso
  finito.
\newblock {\em Ann. Mat. Pura Appl. (4)}, 70:1--201, 1965.

\bibitem{MR2452829}
Ichiro Shimada.
\newblock Transcendental lattices and supersingular reduction lattices of a
  singular {$K3$} surface.
\newblock {\em Trans. Amer. Math. Soc.}, 361(2):909--949, 2009.

\bibitem{MR3166075}
Ichiro Shimada.
\newblock Projective models of the supersingular {$K$}3 surface with {A}rtin
  invariant 1 in characteristic 5.
\newblock {\em J. Algebra}, 403:273--299, 2014.

\bibitem{X56compdata}
Ichiro Shimada.
\newblock A note on a smooth quartic surface containing $56$ lines:
  computational data, 2015.
\newblock http://www.math.sci.hiroshima-u.ac.jp/$\sim$shimada/K3.html.

\bibitem{MR0441982}
T.~Shioda and H.~Inose.
\newblock On singular {$K3$} surfaces.
\newblock In {\em Complex analysis and algebraic geometry}, pages 119--136.
  Iwanami Shoten, Tokyo, 1977.

\end{thebibliography}

\def\cftil#1{\ifmmode\setbox7\hbox{$\accent"5E#1$}\else
  \setbox7\hbox{\accent"5E#1}\penalty 10000\relax\fi\raise 1\ht7
  \hbox{\lower1.15ex\hbox to 1\wd7{\hss\accent"7E\hss}}\penalty 10000
  \hskip-1\wd7\penalty 10000\box7} \def\cprime{$'$} \def\cprime{$'$}
  \def\cprime{$'$} \def\cprime{$'$}

\end{document}